\documentclass[a4paper,12pt,reqno]{amsart}
\usepackage{graphicx}

\usepackage{color}

\usepackage{amsmath,amstext,amssymb,amsopn,amsthm}
\usepackage{url}

\usepackage[margin=30mm]{geometry}
\usepackage{eucal,mathrsfs,dsfont}
\usepackage{soul}

\newtheorem{theorem}{Theorem}[section]
\newtheorem{corollary}[theorem]{Corollary}

\newtheorem{proposition}[theorem]{Proposition}

\theoremstyle{definition}

\theoremstyle{remark}
\newtheorem{remark}[theorem]{Remark}

\definecolor{mr}{rgb}{0.1,0.2,0.7}

\newcommand{\R}{\mathds{R}}

\newcommand{\N}{{\mathds{N}}}

\newcommand{\RR}{\mathrm{I\kern-0.20emR}}
\newcommand{\D}{\mathrm{d}\kern0.2pt}

\newcommand{\vp}{{\varphi}}

\title{Mid-concavity of survival probability for isotropic L{\'e}vy processes}
\author[T. Kulczycki]{Tadeusz Kulczycki}
\thanks{The research was supported in part by NCN grant no. 2011/03/B/ST1/00423.}
\address{Wroc{\l}aw University of Technology, Wyb. Wyspia{\'n}skiego 27, 50-370 Wroc{\l}aw, Poland.}
\email{Tadeusz.Kulczycki@pwr.edu.pl}

\begin{document}
\begin{abstract}
Let $X$ be a symmetric, pure jump, unimodal L{\'e}vy process in $\mathds{R}$ with an infinite L{\'e}vy measure. We prove that for any fixed $t > 0$ the survival probability $P^x(\tau_{(-a,a)} > t)$ is nondecreasing on $(-a,0]$, nonincreasing on $[0,a)$ and concave on $(-a/2,a/2)$, where $a > 0$ and $\tau_{(-a,a)}$ is the first exit time of the process $X$ from $(-a,a)$. We also show a similar statement for sets $(-a,a) \times F \subset \mathds{R}^d$. 
\end{abstract}

\maketitle

\section{Introduction}
The main purpose of this paper is to investigate the monotonicity and concavity properties of the survival probability  for some L{\'e}vy processes in $\R^d$. Let $\tau_D = \inf\{t \ge 0: \, X_t \notin D\}$ be the first exit time of an open, nonempty set $D \subset \R^d$ of the process $X$. We first formulate our result in one-dimensional setting. 

\begin{theorem}
\label{main1}
Let $X$ be a symmetric, pure jump, unimodal L{\'e}vy process in $\R$ with an infinite L{\'e}vy measure. Let $D =(-a,a)$, where $a > 0$. Put $\psi_t^D(x) = P^x(\tau_D > t)$ for $t \ge 0$ and $x \in \R$. Then for any $t > 0$ the function $x \to \psi_t^D(x)$ is nondecreasing on $(-a,0]$, nonincreasing on $[0,a)$ and concave on $(-a/2,a/2)$.
\end{theorem}

The next theorem is the generalization of the above result to higher dimensions.
\begin{theorem}
\label{main2}
Let $X$ be an isotropic, pure jump, unimodal L{\'e}vy process in $\R^d$, $d \ge 2$ with an infinite L{\'e}vy measure. Let $D =(-a,a) \times F$, where $a > 0$ and $F \subset \R^{d -1}$ be a bounded Lipschitz domain. Put $\psi_t^D(x) = P^x(\tau_D > t)$ for $t \ge 0$, $x \in \R^d$ and let $e_1 = (1,0,\ldots,0) \in \R^d$. Then for any $t > 0$ and $\tilde{x} \in \{0\} \times F$ the function $y \to \psi_t^D(y e_1 +\tilde{x})$ is nondecreasing on $(-a,0]$, nonincreasing on $[0,a)$ and concave on $(-a/2,a/2)$.
\end{theorem}

In Section 4 we will apply these results to obtain analogous properties of first eigenfunctions for the related Dirichlet eigenvalue problem.

\begin{remark}
The property that the function $x \to \psi_t^{(-a,a)}(x)$ or $y \to \psi_t^{(-a,a)\times F}(y e_1 +\tilde{x})$ is concave on $(-a/2,a/2)$ is called {\it{mid-concavity}} (see Definition 1.1 in \cite{BKM2006}).
\end{remark}

The above results for isotropic $\alpha$-stable processes in $\R^d$ (where $\alpha \in (0,2]$) and intervals $(-a,a)$ or hyperrectangles $\prod_{i = 1}^d (-a_i,a_i)$ are well known. They were proved by R. Ba{\~n}uelos, T. Kulczycki and P. M{\'e}ndez-Hern{\'a}ndez in \cite{BKM2006}. Indeed, the methods from \cite{BKM2006} allow to extend these results for intervals or hyperrectangles to arbitrary subordinated Brownian motions in $\R^d$. 

The main novelty of the results in this paper is that they concern arbitrary isotropic pure jump, unimodal L{\'e}vy processes in $\R^d$ with an infinite L{\'e}vy measure. 
The method used in the proof of Theorems \ref{main1}, \ref{main2} is completely different than the method used in \cite{BKM2006}. The key idea of the proof of Theorems \ref{main1}, \ref{main2} is probabilistic. A very important step in this proof is the use of some results concerning the so-called {\it{difference processes}} which were introduced in \cite{KR2015}. 

The proof in \cite{BKM2006} is analytical. The main idea in \cite{BKM2006} (for $D = (-a,a)$) is to prove some properties of 
$$
\int_{-a}^a \ldots \int_{-a}^a \prod_{i = 1}^n p_{t_i}(x_{i-1} -x_i) \, dx_1 \ldots dx_n
$$
for gaussian kernels $p_t(x)$ and then use subordination to show monotonicity and midconcavity for
\begin{equation}
\label{iterated}
x \to P^x(X_{t_1} \in D, \ldots,X_{t_n} \in D)
\end{equation}
The results for $P^x(\tau_D > t)$ in \cite{BKM2006} follows by a limiting procedure. 

Note that in this paper we do not study properties of the function (\ref{iterated}) but we study only properties of $P^x(\tau_D > t)$.

Very recently many researchers have been studying convexity properties of solutions of equations involving fractional Laplacians see \cite{BB2015}, \cite{G2014}, \cite{KS2014}, \cite{K2014}, \cite{NR2015}. In particular, concavity properties of the first eigenfunction for the Dirichlet eigenvalue problem on an interval for the fractional Laplacians have been studied in \cite{BB2015} and \cite{KS2014}. In this paper, using a probalistic approach, we obtain concavity properties of the first eigenfunction for the Dirichlet eigenvalue problem on an interval for much more general nonlocal operators, namely generators of the isotropic unimodal L{\'e}vy processes.

The paper is organized as follows. In Section 2 we present notation and collect some known facts needed in the rest of the paper.  Section 3 contains proofs of Theorems \ref{main1}, \ref{main2}. In Section 4 we present regularity results of first eigenfunctions for the related Dirichlet eigenvalue problem.

\section{Preliminaries}

For $x \in \R^d$ and $r > 0$ we let $B(x,r) = \{y \in \R^d: \, |y - x| < r\}$. 
A Borel measure on $\R^d$ is called {\it{isotropic unimodal}} if on $\R^d \setminus \{0\}$ it is absolutely continuous with respect to the Lebesgue measure and has a finite radial, radially nonincreasing density function (such measures may have an atom at the origin). 

A L{\'e}vy process $X = (X_t, t \ge 0)$ in $\R^d$ is called isotropic unimodal if its transition probability $p_t(dx)$ is isotropic unimodal for all $t > 0$.  When additionally $X$ is a pure-jump process then the following L{\'e}vy-Khintchine formula holds for $t > 0$ and $\xi \in \R^d$,
$$
E^0 e^{i\xi X_t} = \int_{\R^d} e^{i\xi x} p_t(dx) = e^{-t \psi(\xi)} \quad \text{where} \quad \psi(\xi) = \int_{\R^d} (1 - \cos(\xi x)) \nu(dx).
$$
$\psi$ is the characteristic exponent of $X$ and $\nu$ is the L{\'e}vy measure of $X$. $E^0$ is the expected value for the process $X$ starting from $0$. Recall that a L{\'e}vy measure is a measure concentrated on $\R^d \setminus \{0\}$ such that $\int_{\R^d} (|x|^2 \wedge 1) \nu(dx) < \infty$. Isotropic unimodal pure-jump L{\'e}vy measures are characterized in \cite{W1983} by unimodal L{\'e}vy measures $\nu(dx) = \nu(x) \, dx = \nu(|x|) \, dx$. 

Unless explicitly stated otherwise in what follows we assume that $X$ is a pure-jump isotropic unimodal L{\'e}vy process in $\R^d$ with (isotropic unimodal) infinite L{\'e}vy measure $\nu$. Then for any $t > 0$ the measure $p_t(dx)$ has a radial, radially nonincreasing density function $p_t(x) = p_t(|x|)$ on $\R^d$ with no atom at the origin. However, it may happen that $p_t(0) = \infty$, for some $t > 0$. As usual, we denote by $P^x$ and $E^x$ the probability measure and the corresponding expectation for the the process starting from $x \in \R^d$.

Let $D \subset \R^d$ be an open, nonempty set.  We define a {\it{killed process}} $X_t^D$ by $X_t^D = X_t$ if $t < \tau_D$ and $X_t^D = \partial$ otherwise, where $\partial$ is some point adjoined to $D$. The transition density for $X_t^D$ on $D$ is given by
\begin{equation}
\label{pdtxy}
p_D(t,x,y) = p_t(x-y) - E^x(p_{t - \tau_D}(X(\tau_D)-y), \, t > \tau_D), \quad x,y \in D, \, t > 0,
\end{equation}
that is for any Borel set $A \subset \R^d$ we have
$$
P^x(X_t^D \in A) = \int_A p_D(t,x,y) \, dy, \quad x \in D, \, t> 0.
$$
We have $p_D(t,x,y) = p_D(t,y,x)$, $x, y \in D$, $t > 0$.
We define the {\it{Green function}} for $X_t^D$ by 
$$
G_D(x,y) = \int_0^{\infty} p_D(t,x,y) \, dt, \quad x,y \in D,
$$
$G_D(x,y) = 0$ if $x \notin D$ or $y \notin D$. 

Let $D \subset \R^d$ be an open, nonempty set. The distribution $P^x(X(\tau_D) \in \cdot)$ is called the {\it{harmonic measure}} with respect to $X$. The harmonic measure for Borel sets $A \subset (\overline{D})^c$ is given by the Ikeda-Watanabe formula \cite{IW1962},
\begin{equation}
\label{IW}
P^x(X(\tau_D) \in A) = \int_A \int_D G_D(x,y) \nu(y-z) \, dy \, dz, \quad x \in D.
\end{equation} 

When $D \subset \R^d$ is a bounded, open Lipschitz set then we have  \cite{Sztonyk2000}, \cite{Millar1975},
\begin{equation}
\label{Xboundary}
P^x(X(\tau_D) \in \partial D) = 0, \quad x \in D.
\end{equation}
It follows that for such sets $D$ the Ikeda-Watanabe formula (\ref{IW}) holds for any Borel set $A \subset D^c$.
Let $D \subset \R^d$ be an open, nonempty set. For any $s > 0$, $x \in D$, $z \in (\overline{D})^c$ put 
\begin{equation}
\label{hformula}
h_D(x,s,z) = \int_D p_D(s,x,y) \nu(y-z) \, dy.
\end{equation}
By the Ikeda-Watanabe formula \cite{IW1962} for any Borel $A \subset (0,\infty)$, $B \subset (\bar{D})^c$ we have 
\begin{equation}
\label{IW2}
P^x(\tau_D \in A, X(\tau_D) \in B) = \int_A \int_B h_D(x,s,z) \, dz \, ds, \quad x \in D.
\end{equation}
If (\ref{Xboundary}) holds then we can take $B \subset D^c$ in (\ref{IW2}). 

\section{The monotonicity and midconcavity}

We will prove both Theorems \ref{main1}, \ref{main2} simultaneously. Let $X$ be an isotropic, pure jump, unimodal L{\'e}vy process in $\R^d$, $d \ge 1$ with an infinite L{\'e}vy measure. Let $D = (-a,a) \times F$, where $a > 0$ and $F \subset \R^{d -1}$ is a bounded Lipschitz domain (or $D = (-a,a)$ when $d = 1$). Put $e_1 = (1,0\ldots,0) \in \R^d$. Note that for any $x \in D$ we have $P^x(X(\tau_D) \in \partial D) = 0$. 

The key point in this section is the following result.
\begin{proposition}
\label{fformula}
Let $U = (b,c) \times F$ (or $U = (b,c)$ when $d = 1$), where $-a \le b < c \le a$. Put $l(U) = b$, $r(U) = c$, $m(U) = (b+c)/2$, $U_- = (b,m(U)) \times F$, $U_+ = (m(U),c) \times F$, $H_-(U) = (-\infty,m(U)) \times \R^{d-1}$, $H_+(U) = (m(U),\infty) \times \R^{d-1}$ (or $U_- = (b,m(U))$, $U_+ = (m(U),c)$, $H_-(U) = (-\infty,m(U))$, $H_+(U) = (m(U),\infty)$ when $d = 1$). For any $x = (x_1,\ldots,x_d) \in \R^d$ let
$$
T_U(x) = x + 2 e_1(m(U) - x_1),
$$
(this is the reflection with respect to the hyperplane $x_1 = m(U)$, or with respect to a point $m(U)$ when $d = 1$). For any $s > 0$, $x \in U_+$, $z \in (\overline{U})^c$ put
$$
f_s^U(x,z) = \int_{U_+} \left(p_U(s,x,y) - p_U(s,T_U(x),y)\right) \left(\nu(y - z) - \nu(T_U(y) - z)\right) \, dy.
$$
For any $x \in U_+$, $t > 0$ we have
\begin{equation}
\label{diffformula}
\psi_t^D(x) - \psi_t^D(T_U(x)) =
\int_{U^c} \int_0^t f_s^U(x,z) \psi_{t - s}^D(z) \, ds \, dz,
\end{equation}
\begin{eqnarray}
\label{positivity}
f_s^U(x,z) &\ge 0& \quad \text{for} \quad s > 0, \, z \in H_+(U) \setminus \overline{U_+},\\
\label{negativity}
f_s^U(x,z) &\le 0& \quad \text{for} \quad s > 0, \, z \in H_-(U) \setminus \overline{U_-}.
\end{eqnarray}
\end{proposition}
\begin{proof}
By the strong Markov property and (\ref{IW2}) for any $x \in U$, $t > 0$ we have
\begin{eqnarray*}
\psi_t^D(x) &=& P^{x}(\tau_D > t)\\
&=& P^{x}(\tau_U > t) + E^x\left(\tau_U \le t, \left[P^{X(\tau_U)}(\tau_D > t-s)\right]_{s = \tau_U}\right)\\
&=& \psi_t^U(x) + \int_U \int_0^t p_U(s,x,y) \int_{U^c} \psi_{t - s}^D(z) \nu(y - z) \, dz \, ds \, dy.
\end{eqnarray*}
It follows that
\begin{eqnarray}
\label{IW3}
&& \psi_t^D(x) -  \psi_t^D(T_U(x)) = \psi_t^U(x) -  \psi_t^U(T_U(x))\\
\label{IW4}
&& + \int_U \int_0^t \left(p_U(s,x,y)-p_U(s,T_U(x),y)\right) \int_{U^c} \psi_{t - s}^D(z) \nu(y - z) \, dz \, ds \, dy.
\end{eqnarray}
For any $x \in U_+$, $t > 0$ by the symmetry of the process $X$ and the definition of $T_U(x)$ we have 
\begin{equation}
\label{psiU}
\psi_t^U(x) =  \psi_t^U(T_U(x)).
\end{equation}
For any $x \in U_+$, $t > 0$, $z \in (\overline{U})^c$ we also have
\begin{eqnarray*}
&& \int_U  \left(p_U(s,x,y)-p_U(s,T_U(x),y)\right) \nu(y - z) \, dy\\
&=& \int_{U_+}  \left(p_U(s,x,y)-p_U(s,T_U(x),y)\right) \nu(y - z) \, dy\\
&& + \int_{U_-}  \left(p_U(s,x,y)-p_U(s,T_U(x),y)\right) \nu(y - z) \, dy\\
&=& \int_{U_+}  \left(p_U(s,x,y)-p_U(s,T_U(x),y)\right) \nu(y - z) \, dy\\
&& + \int_{U_+}  \left(p_U(s,x,T_U(y))-p_U(s,T_U(x),T_{U}(y))\right) \nu\left(T_U(y) - z\right) \, dy\\
&=& \int_{U_+}  \left(p_U(s,x,y)-p_U(s,T_U(x),y)\right) \left(\nu(y - z) - \nu\left(T_U(y) - z\right)\right) \, dy,
\end{eqnarray*}
(where in the last equality we used $p_U(s,T_U(x),T_{U}(y)) = p_U(s,x,y)$). Applying this, (\ref{IW3}-\ref{IW4}) and (\ref{psiU}) we get (\ref{diffformula}).

Note that $p_U(s,x,y)-p_U(s,T_U(x),y)$ is the transition density of the so-called difference process (with respect to the hyperplane $x_1 = m(U)$ or the point $m(u)$ when $d = 1$) killed on exiting $U_+$ (see Section 4 in \cite{KR2015} for more details). By (19) in \cite{KR2015} and the first formula after the proof of Lemma 4.3 in \cite{KR2015} we obtain that $p_U(s,x,y)-p_U(s,T_U(x),y) \ge 0$ for any $s > 0$, $x, y \in U_+$. By unimodality of $\nu(x)$ we obtain that $\nu(y - z) - \nu\left(T_U(y) - z\right) \ge 0$ for any $y \in U_+$, $z \in H_+(U) \setminus \overline{U_+}$ and $\nu(y - z) - \nu\left(T_U(y) - z\right) \le 0$ for any $y \in U_+$, $z \in H_-(U) \setminus \overline{U_-}$. This gives (\ref{positivity}) and (\ref{negativity}).
\end{proof}

Now we will show our main results.

\begin{proof}[proof of Theorems \ref{main1}, \ref{main2}]
First we study monotonicity of $\psi_t^{D}$. Fix $\tilde{x} \in \{0\} \times F$ and $-a < x_1' < x_1'' \le 0$. Put $x' = x_1' e_1 +\tilde{x}$, $x'' = x_1'' e_1 +\tilde{x}$ (or $x' = x_1'$, $x'' = x_1''$ when $d = 1$). Let $b = -a$, $x_* = (x_1' + x_1'')/2$, $c = -a + 2(x_* - (-a)) = a + x_1' + x_1''$, $U = (b,c) \times F$ (or $U = (b,c)$ when $d = 1$). Note that $m(U) = x_*$ and $T_U(x'') = x'$. By (\ref{diffformula}) for any $t > 0$ we get
\begin{eqnarray}
\label{psi1}
\psi_t^D(x'') - \psi_t^D(x') &=& \psi_t^D(x'') - \psi_t^D(T_U(x''))\\
\label{psi2}
 &=& \int_{U^c} \int_0^t f_s^U(x'',z) \psi_{t - s}^D(z) \, ds \, dz\\
\label{psi3}
&=& \int_{D \setminus \overline{U}} \int_0^t f_s^U(x'',z) \psi_{t - s}^D(z) \, ds \, dz.
\end{eqnarray}
Note that $D \setminus \overline{U} = (c,a) \times F$ (or $D \setminus \overline{U} = (c,a)$ when $d = 1$) and $c = a + x_1' + x_1'' > m(U) = (x_1' + x_1'')/2$ so $D \setminus \overline{U} \subset H_+(U) \setminus \overline{U_+}$. This, (\ref{psi1}-\ref{psi3}) and (\ref{positivity}) give $\psi_t^D(x'') \ge \psi_t^D(x')$. It follows that the function $y \to \psi_t^D(y e_1 + \tilde{x})$ (or $y \to \psi_t^D(y)$ when $d = 1$) is nondecreasing on $(-a,0]$. By symmetry of the process $X$ and the domain $D$ the function $y \to \psi_t^D(y e_1 + \tilde{x})$ (or $y \to \psi_t^D(y)$ when $d = 1$) is nonincreasing on $[0,a)$.

Now we will study midconcavity of the function $\psi_t^{D}$.
Fix $\tilde{x} \in \{0\} \times F$ and $-a/2 < x_1' < x_1'' <x_1''' \le 0$ such that $x_1''-x_1' = x_1'''- x_1''$. Put $x' = x_1' e_1 +\tilde{x}$, $x'' = x_1'' e_1 +\tilde{x}$, $x''' = x_1''' e_1 +\tilde{x}$ (or $x' = x_1'$, $x'' = x_1''$, $x''' = x_1'''$ when $d = 1$). As above, let $b = -a$, $x_* = (x_1' + x_1'')/2$, $c = -a + 2(x_* - (-a)) = a + x_1' + x_1''$, $U = (b,c) \times F$ (or $U = (b,c)$ when $d = 1$). We have $m(U) = x_*$ and $T_U(x'') = x'$. Note that $l(U) = -a$ and $r(U) = c = a + x_1'+x_1'' \in (0,a)$ (because $-a/2 < x_1' < x_1'' < 0$). 

Let $v_1 = x_1'' - x_1'$ and $v = v_1 e_1$. Put 
$$
W = U + v.
$$
Note that $W = (b+v_1,c+v_1) \times F$ (or $W = (b+v_1,c+v_1)$ when $d = 1$), $m(W) = m(U) + v_1$, $l(W) = l(U) + v_1$, $r(W) = r(U) + v_1 = a + 2 x_1'' \in (0,a)$. 

We have
$$
T_W(x) = x + 2 e_1(m(W) - x_1).
$$
It follows that
\begin{eqnarray*}
T_W(x''') 
&=& x''' + 2 e_1(m(W) - x_1''')\\
&=& x''' + 2 e_1\left(\frac{x_1''' + x_1''}{2} - x_1'''\right)\\
&=& x''' + e_1(x_1'' - x_1''')\\
&=& x''.
\end{eqnarray*}
Using this and Proposition \ref{fformula} applied to $W$ we get for any $t > 0$
\begin{eqnarray}
\label{psi4}
\psi_t^D(x''') - \psi_t^D(x'') &=& \psi_t^D(x''') - \psi_t^D(T_W(x'''))\\
\label{psi5}
 &=& \int_{W^c} \int_0^t f_s^W(x''',z) \psi_{t - s}^D(z) \, ds \, dz.
\end{eqnarray}
Note that $W_+ = U_+ + v$ and $x''' = x'' + v$. Using this and the definition of $f_s^W(x,z)$ we get for any $s > 0$,  $z \in W^c$
\begin{eqnarray}
\label{f1}
&& f_s^W(x''',z)\\ 
\label{f2}
&=& \int_{W_+} \left(p_W(s,x''',y) - p_W(s,T_W(x'''),y)\right) \left(\nu(y - z) - \nu(T_W(y) - z)\right) \, dy\\
\label{f3}
&=& \int_{U_+ + v} \left(p_W(s,x'' + v,y) - p_W(s,T_W(x'' + v),y)\right)\\
\label{f4}
&&\quad \quad\quad \quad \quad \quad \quad \quad \quad \quad \quad \quad \quad \quad \times \left(\nu(y - z) - \nu(T_W(y) - z)\right) \, dy.
\end{eqnarray}
Using substitution $q = y - v$ this is equal to 
\begin{eqnarray}
\label{f5}
&&\int_{U_+} \left(p_W(s,x'' + v,q + v) - p_W(s,T_W(x'' + v), q + v)\right)\\
\label{f6}
&&\quad \quad\quad \quad \quad \quad \quad \quad \quad \quad \times \left(\nu(q + v - z) - \nu(T_W(q + v) - z)\right) \, dq.
\end{eqnarray}
For any $s > 0$,  $q \in U_+$ we have 
$$
p_W(s,x'' + v,q + v) = p_{U + v}(s,x'' + v,q + v) = p_U(s,x'',q ).
$$
By the definition of $T_W$ and the equality $m(W) = m(U) + v_1$ we get
$$
T_W(x''+v) = x'' + v + 2 e_1(m(W) - x_1'' - v_1) 
= x'' + 2 e_1(m(U) - x_1'') + v = T_U(x'') + v.
$$
Hence for any $s > 0$,  $q \in U_+$ we obtain
$$
p_W(s,T_W(x'' + v),q + v) = p_{U + v}(s,T_U(x'') + v,q + v) = p_U(s,T_U(x''),q ).
$$
By similar arguments as above for any $q \in U_+$ we get $T_W(q + v) = T_U(q) + v$. Hence for any $q \in U_+$ and $z \in W^c$ we obtain
$$
\nu(q + v - z) - \nu(T_W(q + v) - z) = \nu(q - (z - v)) - \nu(T_U(q) - (z - v)).
$$
Using this, (\ref{f1}-\ref{f4}) and (\ref{f5}-\ref{f6}) we get for any $s > 0$, $z \in (\overline{W})^c$
\begin{eqnarray*}
&& f_s^W(x''',z)\\
&=& \int_{U_+} \left(p_U(s,x'',q) - p_U(s,T_U(x''), q)\right) \left(\nu(q - (z - v)) - \nu(T_U(q) - (z - v))\right) \, dq\\
&=& f_s^U(x'',z - v).
\end{eqnarray*}
Using this, the fact that $W^c = U^c + v$ and (\ref{psi4}-\ref{psi5}) we get for any $t > 0$
\begin{eqnarray}
\label{psi6}
\psi_t^D(x''') - \psi_t^D(x'')
 &=& \int_{W^c} \int_0^t f_s^U(x'',z - v) \psi_{t - s}^D(z) \, ds \, dz\\
 \label{psi7}
 &=& \int_{U^c} \int_0^t f_s^U(x'',z) \psi_{t - s}^D(z + v) \, ds \, dz.
\end{eqnarray}
Put 
$$
L(U) = (-\infty,l(U)) \times F \subset H_-(U), \quad R(U) = (r(U),\infty) \times F \subset H_+(U),
$$
(or $L(U) = (-\infty,l(U)) \subset H_-(U)$, $R(U) = (r(U),\infty) \subset H_+(U)$ when $d = 1$). By (\ref{psi1}-\ref{psi2}) and (\ref{psi6}-\ref{psi7}) we get for any $t > 0$
\begin{eqnarray*}
&& \psi_t^D(x'') - \psi_t^D(x')\\ 
&=& \int_{L(U)} \int_0^t f_s^U(x'',z) \psi_{t - s}^D(z) \, ds \, dz 
+ \int_{R(U)} \int_0^t f_s^U(x'',z) \psi_{t - s}^D(z) \, ds \, dz\\
&=& \text{I} + \text{II},
\end{eqnarray*}
\begin{eqnarray*}
&&\psi_t^D(x''') - \psi_t^D(x'')\\
 &=& \int_{L(U)} \int_0^t f_s^U(x'',z) \psi_{t - s}^D(z + v) \, ds \, dz 
+ \int_{R(U)} \int_0^t f_s^U(x'',z) \psi_{t - s}^D(z + v) \, ds \, dz\\
&=& \text{III} + \text{IV}.
\end{eqnarray*}
Since $l(U) = -a$ we get $\text{I} = 0$. Since $L(U) + v \subset H_-(U)$ and by (\ref{negativity}) $f_s^U(x'',z) \le 0$ for $z \in H_-(U)$ we get $\text{III} \le 0$. Recall that $r(U) > 0$ so monotonicity of $y \to \psi_t^D(y e_1 + \tilde{x})$ (or $y \to \psi_t^D(y)$ when $d = 1$) implies that for any $z \in R(U)$, $t > 0$, $s \in (0,t)$ we have $\psi_{t - s}^D(z + v) \le \psi_{t - s}^D(z)$ so $\text{IV} \le \text{II}$. Hence for any $t > 0$ we get
$$
\psi_t^D(x'') - \psi_t^D(x') \ge \psi_t^D(x''') - \psi_t^D(x'').
$$
Recall that $-a/2 < x_1' < x_1'' < x_1''' \le 0$, where $x_1'' - x_1' = x_1''' - x_1''$. Since $x_1'$, $x_1'''$ could be chosen arbitrarily we get that 
$y \to \psi_t^D(y e_1 + \tilde{x})$ (or $y \to \psi_t^D(y)$ when $d = 1$) is concave on $(-a/2,0]$. By the symmetry we obtain that $y \to \psi_t^D(y e_1 + \tilde{x})$ (or $y \to \psi_t^D(y)$ when $d = 1$) is concave on $[0,a/2)$.
\end{proof}

\section{The shape of the first eigenfunction}

Let us recall that $X$ is a pure-jump isotropic unimodal L{\'e}vy process in $\R^d$ with an infinite L{\'e}vy measure $\nu$, $\psi$ is the characteristic exponent of $X$ and $p_t$ is its transition density. In this section we additionally assume that 
\begin{equation}
\label{logpsi}
\lim_{|x| \to \infty} \frac{\psi(x)}{\log |x|} = \infty.
\end{equation}
This guarantees that for any $t > 0$ the function $p_t$ is continuous and bounded on $\R^d$. 

Let $D \subset \R^d$ be a bounded open set. The condition (\ref{logpsi}) and formula (\ref{pdtxy}) imply that for any fixed $t > 0$, $x \in D$ the function $y \to p_D(t,x,y)$ is continuous on $D$. Since $p_{D}(t,x,y) = p_D(t,y,x)$, $t > 0$, $x, y \in D$ we obtain that for any fixed $t > 0$, $y \in D$ the function $x \to p_D(t,x,y)$ is continuous on $D$. The transition operator $P_t^D$ for the killed process $X_t^D$ is defined by 
$$
P_t^D f(x) = \int_D p_D(t,x,y) f(y) \, dy, \quad x \in D, \, t > 0.
$$

Now we introduce the Dirichlet eigenvalue problem on $D$ for the L{\'e}vy process $X$. Such problem is well known in the literature see e.g. \cite{CS2005}. $\{P_t^D\}_{t \ge 0}$ forms a strongly continuous semigroup on $L^2(D)$. Since $p_D(t,x,y) \le p_t(x - y)$, $\|p_t\|_{\infty} < \infty$ and $D \subset \R^d$ is bounded we obtain that for any $t > 0$ the operator $P_t^D$ is a Hilbert-Schmidt operator. From the general theory of semigroups there exists an orthonormal basis $\{\vp_n\}_{n = 1}^{\infty}$ in $L^2(D)$ and a corresponding sequence 
$$
0 < \lambda_1 < \lambda_2 \le \lambda_3 \le \ldots, \quad \quad \quad \lim_{n \to \infty} \lambda_n = \infty,
$$
such that for any $n \in \N$, $t >0$, $x \in D$ we have
\begin{equation}
\label{spectral}
P_t^D \vp_n(x) = e^{-\lambda_n t} \vp_n(x).
\end{equation}
$\lambda_1$ has multiplicity one and we may assume that $\vp_1 > 0$ on $D$. By properties of $p_D(t,x,y)$ all eigenfunctions $\vp_n$ are bounded and continuous on $D$. It is well known that
$$
p_D(t,x,y) = \sum_{n = 1}^{\infty} e^{-\lambda_n t} \vp_n(x) \vp_n(y), \quad t > 0, \, x,y \in D.
$$
It follows that for any $t > 0$ and $x \in D$ we have
\begin{eqnarray*}
P^x(\tau_D > t) 
&=& \int_D p_D(t,x,y) \, dy \\
&=& \sum_{n = 1}^{\infty} e^{-\lambda_n t} \vp_n(x) \int_D \vp_n(y) \, dy.
\end{eqnarray*}
Hence for any $x \in D$ we have
$$
\lim_{t \to \infty} e^{\lambda_1 t} P^x(\tau_D > t) = \vp_1(x) \int_D \vp_1(y) \, dy.
$$
Using this and Theorems \ref{main1}, \ref{main2} we immediately obtain the following results.
\begin{corollary}
\label{cor1}
Let $X$ be a symmetric, pure jump, unimodal L{\'e}vy process in $\R$ satisfying (\ref{logpsi}) with an infinite L{\'e}vy measure. Let $D =(-a,a)$, where $a > 0$. Let $\vp_1$ be the first eigenfunction of the spectral problem (\ref{spectral}) on $D$ for the process $X$. Then for any $t > 0$ the function $x \to \vp_1(x)$ is nondecreasing on $(-a,0]$, nonincreasing on $[0,a)$ and concave on $(-a/2,a/2)$.
\end{corollary}

\begin{corollary}
\label{cor2}
Let $X$ be an isotropic, pure jump, unimodal L{\'e}vy process in $\R^d$, $d \ge 2$ satisfying (\ref{logpsi}) with an infinite L{\'e}vy measure. Let $D =(-a,a) \times F$, where $a > 0$ and $F \subset \R^{d -1}$ be a bounded Lipschitz domain. Let $\vp_1$ be the first eigenfunction of the spectral problem (\ref{spectral}) on $D$ for the process $X$. Let $e_1 = (1,0,\ldots,0) \in \R^d$. Then for any $t > 0$ and $\tilde{x} \in \{0\} \times F$ the function $y \to \vp_1(y e_1 +\tilde{x})$ is nondecreasing on $(-a,0]$, nonincreasing on $[0,a)$ and concave on $(-a/2,a/2)$.
\end{corollary}

\end{document}